\documentclass[12pt]{amsart} 
\usepackage[margin=1.2in]{geometry}
\usepackage{enumerate}
\usepackage{pgf,tikz}
\usetikzlibrary{arrows.meta}

\usepackage{pgfplots}
\usetikzlibrary{automata}
\usetikzlibrary{arrows} 
\usetikzlibrary{positioning}
\usetikzlibrary{snakes}

\usepackage{float}
\usepackage{graphicx}
\usepackage{bm}
\usepackage{multirow}
\usepackage{graphicx}
\usepackage[numbers]{natbib}
\usepackage{amsmath}
\usepackage{placeins}
\usepackage{amssymb}
\usepackage{epstopdf}
\usepackage{pdfpages}
\usepackage{amsthm}
\usepackage{xcolor}
\usepackage{setspace}
\usepackage{dsfont}

\usepackage{tikz-cd} 
\usepackage{adjustbox}

\newcommand{\inner}[2]{\langle {#1} ,{#2} \rangle }

\newtheorem{theorem}{Theorem}[section]

\newtheorem{corollary}[theorem]{Corollary}

\newtheorem{definition}[theorem]{Definition}
\newtheorem{remark}[theorem]{Remark}

\newcommand*\E{\mathop{}\!\mathbb{E}}
\newcommand{\Exp}{\mathds{E}}

\newcommand{\ii}{\mbox{i}}
\newcommand{\Prob}{\mathop{}\!\mathbb{P}}

\newcommand{\vect}[1]{\vec{#1}}

\renewcommand{\vect}[1]{\boldsymbol{#1}}

\newcommand{\mat}[1]{\boldsymbol{#1}}

\newcommand*\dd{\mathop{}\!\mathrm{d}}

\usepackage{multicol}

\author[H. \smash{Albrecher}]{Hansj\"org Albrecher}
\address[Hansj\"org Albrecher]{ Department of Actuarial Science, Faculty of Business and Economics and Swiss Finance Institute, University of Lausanne, CH-1015 Lausanne, Switzerland}

\email{{hansjoerg.albrecher@unil.ch}}

\author[M. \smash{Bladt}]{Martin Bladt}
\address[Martin Bladt]{Department of Actuarial Science, Faculty of Business and Economics, University of Lausanne, CH-1015 Lausanne, Switzerland}

\email{{martin.bladt@unil.ch}}

\author[M. \smash{Bladt}]{Mogens Bladt}
 \address[Mogens \smash{Bladt}]{Department of Mathematical Sciences, University of Copenhagen, Universitetsparken 5, DK-2100 Copenhagen \O, Denmark}
\email{bladt@math.ku.dk}

\title{Multivariate fractional phase--type distributions}
\date{Note, \today}

\begin{document}

\begin{abstract}We extend the Kulkarni class of multivariate phase--type distributions in a natural time--fractional way to construct a new class of multivariate distributions with heavy-tailed Mittag-Leffler(ML)-distributed marginals. The approach relies on assigning rewards to a non--Mar\-ko\-vi\-an jump process with ML sojourn times. This new class complements an earlier multivariate ML construction \cite{multiml} and in contrast to the former also allows for tail dependence. We derive properties and characterizations of this class, and work out some special cases that lead to explicit density representations. 
\end{abstract} 
\maketitle

\section{Introduction}
The formulation of flexible and at the same time parsimonious models for stochastic phenomena is a crucial ingredient in the process of managing risks in various application areas of operations research. On the one hand, a given set of data should be represented reasonably well when putting them into the frame of a calibrated model (and finally replacing them by the latter for the further purposes in the risk analysis). Yet, on the other hand, one needs to avoid overfitting of data and resulting lack of robustness of fitted parameters when applied to updated data sets. In addition, in quite a number of situations (notably in quantitative risk management, see e.g.\ \cite{McN}) models are used to extrapolate beyond the range of existing data, and then capturing the main pattern is essential, but overfitting can lead to wrong conclusions about tail properties, particularly in higher dimensions, see also \cite{beirlant2006}. Another important aspect in this context is that it is quite useful, if models still allow for explicit densities or expressions for the relevant intended measures of risk. This leads to more efficient fitting procedures, and particularly allows to study sensitivities with respect to changes in model parameters in a more explicit way. \\
In this context, it is quite attractive to have a set of models that a priori is quite general and versatile, but then in the process of fitting the model to actual data reduces to a simpler model in some nested way, if the given data suggest that. One classical example of such a class of models in one dimension are the phase-type distributions (originally introduced by Neuts \cite{neuts75}), which builds upon the simplicity of an exponential distribution, but then concatenates exponential ingredients by considering the absorption time of a homogeneous Markov jump process on transient states (phases) into one absorption state with, if needed, many phases and arbitrary intensity matrix (the exponential being the special case of one transient phase only). The gained flexibility is enormous, as the resulting class of phase-type distributions can be shown to be dense (in the sense of weak convergence) in the class of all distributions on the positive half-line (see e.g. \cite{asmner}). However, the resulting model will only be parsimonious if the underlying risk is close to an exponential structure (e.g. in the tail), as otherwise the number of phases needed for a good fit will be excessive. Yet, on the computational level, the class of phase-type distributions is pleasant, as it can be understood as an (almost exhaustive) subclass of matrix-exponential distributions (that is, an exponential distribution with matrix parameter), for which explicit calculations are available (see e.g.\ \cite{bladt2017matrix}). If the underlying risk has a tail heavier than exponential, then it was recently shown in \cite{ab18inh} that extending the above construction principle to time-inhomogeneous Markov jump processes, adapts the fitting procedure to be built upon other than exponential random variables (namely transforms thereof), and thereby keeps the number of necessary parameters for a good fit very low (essentially leading to matrix-valued parameters of the new base distribution, like Pareto or Weibull). See also \cite{BladtNandayapa2018} for another alternative to modelling heavy-tailed data within the phase-type paradigm. Finally, in \cite{albrecher2019matrix} a random time transformation (based on a stable($\alpha$) random variable with $0<\alpha\le 1$) in the underlying Markov jump process was considered, which leads to a  Mittag-Leffler (ML) distribution as the base distribution, and a resulting flexible family of ML distributions with matrix argument (which later will be referred to as the fractional phase-type class PH$_{\alpha}$). The latter is typically heavy-tailed, but contains the phase-type distributions as the limiting special case $\alpha=1$. Hence the data fitting procedure can decide on which type of model is most suitable for a given data set.\\

For modelling in more than one dimension, Kulkarni \cite{Kulkarni:1989ti} formulated a multivariate version $\text{MPH}^*$ of the phase-type construction by having each component of a random vector collecting different rewards in every state of the (common) Markov jump process, thereby creating possibly dependent phase-type random variables, whose joint Laplace transform is still fully explicit. It could be shown that the resulting family of distributions is again dense in the class of all distributions on the positive orthant. In \cite{multiml}, this multivariate construction was extended to define a transparent class of multivariate generalized matrix ML (GMML) distributions by applying an independent stable($\alpha_i$) random time transformation to each component of the Kulkarni construction. Mathematically, this amounts to a replacement of each argument $\theta_i$ in the joint Laplace transform by its power $\theta_i^{\alpha_i}$, leading to explicit expressions for a number of particular cases (see \cite{multiml} for details). An unfortunate consequence of this procedure is that the resulting multivariate model is necessarily (asymptotically) tail-independent. This can also be seen from an alternative interpretation of the above resulting random vector as the one obtained from stopping each component of a multivariate stable$(\alpha_i)$ L\'{e}vy process (with independent components, cf.\ \cite{kyprianou2006introductory}) at the (dependent) multivariate phase-type times from the Kulkarni class. However, in many applications one observes possible dependence in the tails, and a proper modelling of that tail dependence is a particular concern in risk management. \\

In this paper, we propose another way to extend Kulkarni's multivariate phase-type class to formulate a new class $\text{MPH}^*_{\alpha}$ of  multivariate Mittag-Leffler distributions that does allow for tail dependence. Concretely, we return to the interpretation of a matrix Mittag-Leffler distributed random variable as the absorption time of a finite state-space semi-Markov process with (state-dependent) ML distributed sojourn times and one absorbing state, see \cite{albrecher2019matrix}. This involves the consideration of Kolmogorov forward equations with fractional derivates of order $\alpha$. We then impose the reward structure element of Kulkarni's multivariate construction on this semi-Markov process. Interestingly, the joint Laplace transform of the resulting random vector is again explicit, and on the analytical side differs from the one of the construction in \cite{multiml} merely by the fact that the power $\alpha$ is applied to the scalar product of each reward vector and the vector of Laplace arguments rather than to the Laplace arguments themselves (with the additional restriction that the value for $\alpha$ in each component now has to be the same). This approach leads to an attractive complement candidate for the modelling of multivariate matrix Mittag-Leffler distributions which allows for dependence in the tail. In a way, the present approach naturally extends Kulkarni's approach onto the appropriate more general semi-Markovian process governed by fractional Kolmogorov forward equations. As compared to the approach in \cite{multiml}, the stretching of time is here applied continuously until absorportion, rather than only on the final absorption times, allowing for a different degree of flexibility in the fine structure of the dependence modelling across the different random components. Figure \ref{mapofdists} depicts the relation between the respective models in the literature, and highlights the fact that the $\text{MPH}^*_{\alpha}$ class proposed here is a natural next step from a conceptual point of view.

\begin{figure}[h]
\includegraphics[width=14.5cm,trim=.5cm .5cm .5cm .5cm,clip]{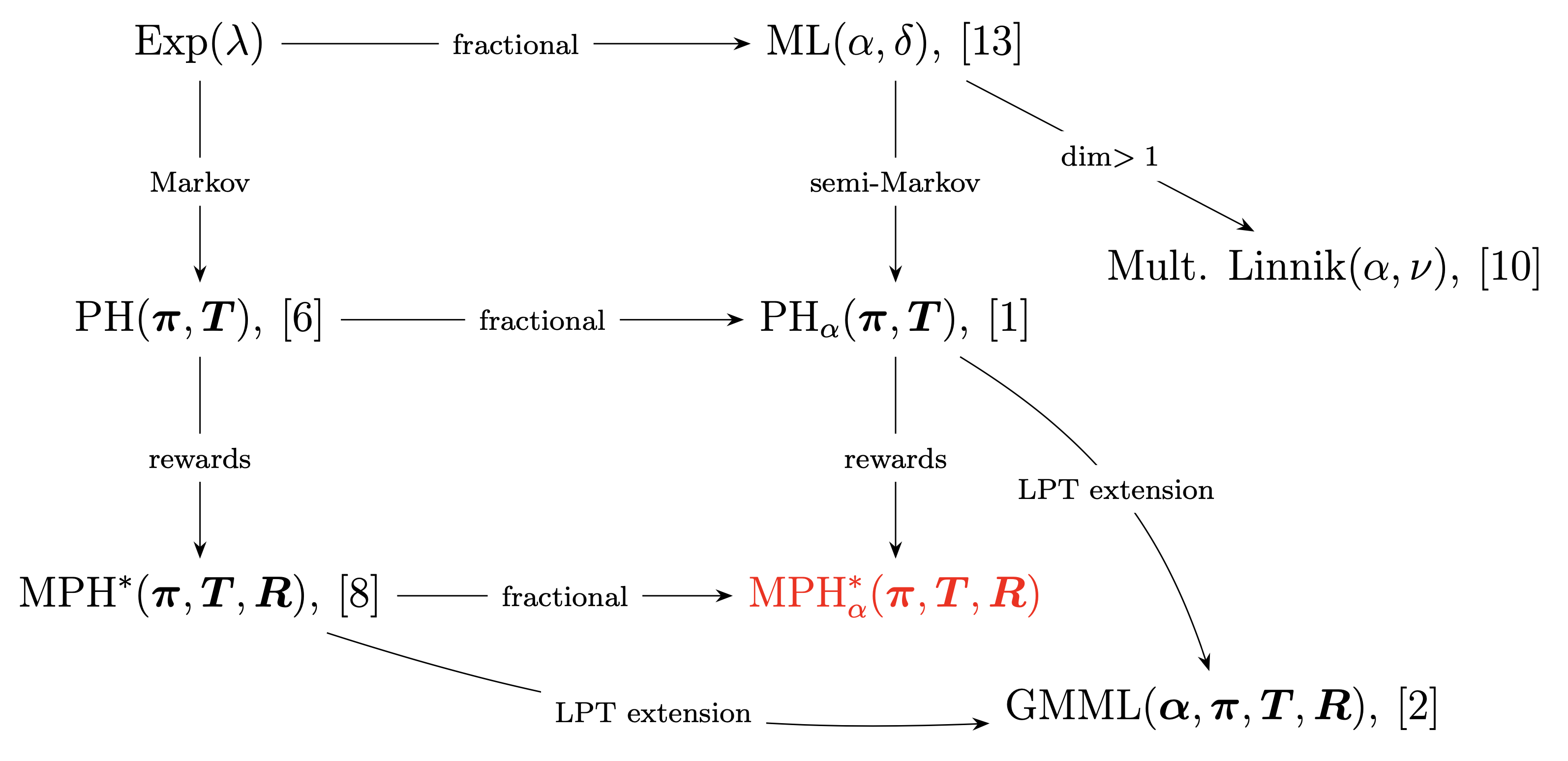}
	\caption{Schematic representation of distributions related to multivariate fractional phase--type distributions. Each arrow indicates a generalization.}\label{mapofdists}
\end{figure}

The remainder of the paper is organized as follows. In Section \ref{ren_sec_back} we review some relevant background on phase-type and (matrix) Mittag-Leffler distributions. Section \ref{ren_fractph_sec} develops the class $\text{MPH}_\alpha^\ast$ of multivariate fractional phase-type distributions as a reward-based multivariate construction using a time-fractional sample path approach with matrix ML distributed marginals. It is shown that this new class (as well as its extension to powers) is itself dense among all distributions on the positive orthant in several ways, and a characterization in terms of a product representation is provided. Finally, it is shown that any linear combination of the random components is again matrix ML distributed (possibly with an additional atom at zero). In Section \ref{sec:spec} we illustrate two particular cases that lead to explicit density representations. Section \ref{concl} concludes the paper.

\section{Background}\label{ren_sec_back}

\subsection{Phase--type distributions (PH)}
Consider a state space $E=\{1,2,\dots, p,p+1\}$, and a Markov jump process $\{X_t\}_{t\ge 0}$ evolving on $E$ such that the first $p$ states are transient and the state $p+1$ is absorbing. The intensity matrix of such a process has the form
\begin{equation*}
 \mat{\Lambda} 
=
\begin{pmatrix}
\mat{T} & \vect{t} \\
\vect{0} & 0 
\end{pmatrix},
\end{equation*}
where $\mat{T}$ is a sub-intensity matrix of dimension $p\times p$, consisting of jump rates between the transient states rates. We further specify an initial distribution, concentrated on the transient states $1,...,p$, by $\pi_k=\Prob(X_0=k)$ for $k=1,\dots, p.$ Thus, if we write $\vect{\pi}=(\pi_1,\dots,\pi_p)$, we have that $\vect{\pi}\vect{e}=1$, where $\vect{e}$ is the $p$-dimensional column vector of $1$'s. We also write by convention $$\vect{t}=-\mat{T}\vect{e},$$
which is a column vector whose elements are the intensities of jumping to the absorbing state. 
A phase--type distribution is defined as the absorption time of $X_t$, that is, if we let
$$\tau=\inf\{t>0|X_t=p+1\},$$
we say that $\tau$ follows a phase--type distribution with parameters $\vect{\pi},\mat{T}$, and write $\tau\sim \text{PH}(\vect{\pi},\mat{T})$. In general, the parametrization is non-identifiable, in the sense that several initial vectors and sub-intensity matrices can result in the same distribution.

The density and distribution function of $\tau\sim \text{PH}(\vect{\pi},\mat{T})$ are given by
\begin{align*}
f(x)&=\vect{\pi} e^{\mat{T}x}\vect{t}, \quad x>0,\\
F(x)&=1-\vect{\pi}e^{\mat{T}x}\vect{e},\quad x>0,\\
\end{align*}
where the exponential of a matrix $M$ is defined by the formula
$$\exp(\mat{M})=\sum_{n=0}^\infty \frac{\mat{M}^n}{n!} .$$
The Laplace transform is given by
\begin{align}
L(u)=\vect{\pi}(u\mat{I}-\mat{T})^{-1}\vect{t},\label{lu1}
\end{align}
and is always a rational function, well defined for $u>\mbox{Re}(\lambda_{m})$, where $\mbox{Re}$ denotes the real part and where $\lambda_m$ is the eigenvalue of $\mat{T}$ with largest real part, and $\mat{I}$ denotes the identity matrix.

The class of phase--type distributions is closed both under mixing and convolution, which means that also Erlang distributions, Coxian distribution and mixtures thereof are PH distributions. The class is also dense in the class of all distributions on the positive real line (in the sense of weak convergence). This means that any distribution with support on $\mathbb{R}_+$ may be approximated arbitrarily well by a
phase--type distribution (of sufficiently high dimension).


\subsection{Multivariate phase--type distributions (MPH$^*$)}
The class of MPH$^*$ was originally introduced in \cite{Kulkarni:1989ti} and is constructed as follows. Let $\tau \sim \mbox{PH}(\vect{\pi},\mat{T})$ and let $\{ X_t\}_{t\geq 0}$ be the underlying Markov jump. For $i=1,...,n$, let $\vect{r}_k = (r_{1k},r_{2k},...,r_{pk})^\prime$ (column vector) and
define
\[ Y_k = \int_0^\tau \sum_{i=1}^p r_{ik}1\{ X_t=i \} dt,  \ \ k=1,...,n .   \]
If we interpret $r_{ik}$ as the reward rate earned by the process $X_t$ when it is in state $i$, then $Y_k$ is the total amount of reward earned according to $\vect{r}_{k}$ prior to absorption. Let $\mat{R}$ denote the $p\times n$ matrix
\[ \mat{R} = ( \vect{r}_1,...,\vect{r}_n) . \] 
whose the columns consist of the different reward rates leading to the variables $Y_1,...,Y_n$. Then we say that 
$\vect{Y}=(Y_1,...,Y_n)$ has a multivariate distribution of the MPH$^*$ type and we write
 $\vect{Y}\sim \text{MPH}^\ast(\vect{\pi},\mat{T},\mat{R})$.
The multivariate Laplace transform of $\vect{Y}\sim \text{MPH}^\ast(\vect{\pi},\mat{T},\mat{R})$ is given by 
\begin{align}\label{eq:Kulkarni_r}
\E(e^{-\langle \vect{Y},\vect{\theta}\rangle})=\vect{\pi}(\mat{\Delta}(\mat{R}\vect{\theta})-\mat{T})^{-1}\vect{t},
\end{align}
where $\mat{\Delta}(\vect{v})$ denotes the diagonal matrix which has $\vect{v}$ as diagonal.

Multivariate phase--type distributions are dense on $\mathbb{R}_+^n$, and the marginals and their linear combinations are univariate phase--type distributions, which make them a very flexible and attractive class of distributions for statistical as well as non-statistical applications. However, statistical fitting of this class is still in an experimental stage, since the main dimensionality difficulties of the univariate case are exacerbated with the introduction of the additional parameters of $\mat{R}$. 

We refer the reader to \cite{bladt2017matrix} for a recent comprehensive text on phase--type distributions, both in the uni-- and multivariate cases.

\subsection{Univariate fractional phase--type distributions (PH$_\alpha$)}

A Mittag-Leffler (ML) distribution \cite{pillai} has a density of the form
\begin{equation}\label{mldens}
  f_{\lambda,\alpha}(x)
=\lambda x^{\alpha-1} E_{\alpha,\alpha}(-\lambda x^\alpha) ,\ \ \ \ \ \lambda>0,\ 0<\alpha\leq 1, \end{equation}
where
\[   E_{\alpha, \beta}(z)=\sum_{k=0}^{\infty} \frac{z^{k}}{\Gamma(\alpha k+\beta)}, \ \ \beta\in\mathbb{R}, \ \alpha >0  \]
is the so--called Mittag--Leffler function, and we denote the corresponding class by $\mbox{ML}(\alpha,\lambda)$. Note that Pillai's definition \cite{pillai} of the ML distribution is recovered with $\rho_i^{-\alpha} = \lambda_i$.
For $\alpha=1$, \eqref{mldens} reduces to the density of an exponential random variable. 
Recently, in \cite{albrecher2019matrix}, a matrix version of the ML distribution with Laplace transform \begin{equation}
\vect{\pi}(u^\alpha \mat{I}- \mat{T})^{-1}\vect{t},  \quad 0<\alpha\le 1,\label{lu2}
\end{equation}
was introduced, which for $\alpha=1$ reduces to the one of a phase--type distribution (cf. \eqref{lu1}). For scalar $\mat{I}$ and $\mat{T}$ one recovers the classical ML distribution. While the class of distributions with Laplace transform \eqref{lu2} was referred to as a \textit{matrix ML} distribution in \cite{albrecher2019matrix}, we suggest to assign to it the additional name {\em fractional phase-type distribution} (PH$_\alpha(\vect{\pi},\:\vect{T}))$), as this will lead to a simple and somewhat more consistent nomenclature in the sequel. As shown in \cite{albrecher2019matrix}, the density and distribution function are given by
\begin{align*} 
f(x)&=x^{\alpha-1}\vect{\pi}\,E_{\alpha,\alpha}\left(\mat{T}x^\alpha \right)\,\vect{t},\\
F(x)&=1-\vect{\pi}E_{\alpha,1}\left(\mat{T}x^\alpha \right)\vect{e},
\end{align*}
where
\begin{align}\label{cauchyintegralML}
E_{\alpha, \beta}(\vect{T}x^\alpha)=\sum_{k=0}^{\infty} \frac{\vect{T}^{k}x^{\alpha k}}{\Gamma(\alpha k+\beta)}=\frac{1}{2\pi \ii}\oint_\gamma E_{\alpha,\beta}(zx^\alpha )(z\mat{I}-\mat{T})^{-1}\dd z ,
\end{align}
with $\gamma$ denoting a simple path enclosing the eigenvalues of $\mat{T}$. For $X\sim$PH$_\alpha(\vect{\pi},\:\vect{T})$ we have the product representation
\begin{align}\label{prodrepresentation_r}
X\stackrel{d}{=}W^{1/\alpha}S_\alpha,
\end{align}
where $W\sim \mbox{PH}(\vect{\pi},\:\vect{T})$, and $S_\alpha$ is an independent positive stable random variable, cf. \cite{albrecher2019matrix}. Note again that for  $\alpha=1$ we obtain the PH distributions as a special case. 


\section{Multivariate fractional phase--type distributions}\label{ren_fractph_sec}
\subsection{The construction}
Following \cite{albrecher2019matrix}, we begin by constructing a semi-Markov process which has an absorption time given by a PH$_\alpha$ distribution. Let $E=\{1,2,...,p,p+1\}$ be the state space  and
let $\mat{Q}=\{ q_{ij} \}_{i,j\in E}$ denote the transition matrix of a Markov chain $\{ Y_n \}_{n\in \mathbb{N}}$ on $E$, where the first $p$ states are transient and state $p+1$ is absorbing. This means that $\{ Y_n\}_{n\in\mathbb{N}}$ has a transition matrix of the form
\[   \mat{Q} = \begin{pmatrix}
\mat{Q}^1 & \vect{q}^1 \\
\vect{0} & 1 
\end{pmatrix}. \] 
We assume that $q_{ii}=0$ for all $i\neq p+1$. This chain will be the embedded Markov chain in a Markov renewal process with Mittag-Leffler distributed holding times defined below. Let $\alpha \in (0,1]$ and $\lambda_i>0$. For the states $i=1,...,p$, let $T^i_n$, $n=1,2,...$ be independent $\mbox{ML}(\alpha,\lambda_i)$--distributed random variables.
Let
\[ S_n = \sum_{i=1}^n T^{Y_i}_i , \ \ n\geq 1 ,\]
and $S_0=0$.
Define then the semi--Markov process 
\begin{equation}
  X_t = \sum_{n=1}^\infty Y_{n-1} 1\{ S_{n-1}\leq t <S_n  \}  . \label{def:X-process_r},\; t\ge 0.
\end{equation}
The interpretation is that $\{ X_t\}_{t\geq 0}$ jumps between states according to the dynamics of the Markov chain $Y_n$, and $S_n$ denotes the time of the $n$'th jump. The holding time in state $i<p+1$ is $\mbox{ML}(\alpha,\lambda_i)$. The construction is schematically shown in Figure \ref{fig:X-process}.

\begin{figure}[H]
\begin{tikzpicture}[scale=0.75,domain=-1:14]
\draw[->] (0,0)--(14.0,0) node[right] {$t$};
\draw[-] (0,0) --(0,2);
\draw[-,dashed] (0,2)--(0,3);
\draw[->] (0,3)--(0,4) node[above] {$X_t$};
\foreach \y/\ytext in {0.5/1, 1/2, 1.5/3,3.5/p}
\draw[shift={(0,\y)}] (-2pt,0pt) -- (2pt,0pt) node[left] {$\ytext$};

\draw[color=red,very thick,domain=0:1.90] plot (\x,{1.5});
\draw[color=red,very thick,domain=2.1:4.4] plot (\x,{0.5});
\draw[color=red,very thick,domain=4.5:7.9] plot (\x,{3.5});
\draw[color=red,very thick,domain=8.1:10.9] plot (\x,{0.5});
\draw[color=red,very thick,domain=11.1:13] plot (\x,{1.0});

\draw[color=red] (2,1.5) circle (3pt); 
\draw[color=red,fill] (2,0.5) circle (3pt); 
\draw[color=red] (4.5,0.5) circle (3pt); 
\draw[color=red,fill] (4.5,3.5) circle (3pt); 
\draw[color=red] (8.0,3.5) circle (3pt); 
\draw[color=red,fill] (8,0.5) circle (3pt); 
\draw[color=red] (11.0,0.5) circle (3pt); 
\draw[color=red,fill] (11.0,1.0) circle (3pt);

\draw[thick, snake=brace,segment aspect=0.5] (0,1.7) -- (2,1.7);
\draw[thick] (1.5,2.0) node[above,rotate=0] {$\sim {f_{3}}$};

\draw[thick, snake=brace,segment aspect=0.5] (2,0.7) -- (4.5,0.7);
\draw[thick] (3.75,1.0) node[above,rotate=0] {$\sim {
    f_{1}}$};

\draw[thick, snake=brace,segment aspect=0.5] (4.5,3.7) -- (8,3.7);
\draw[thick] (6.8,4.0) node[above,rotate=0] {$\sim {
    f_{p}}$};

\draw[thick, snake=brace,segment aspect=0.5] (8,0.7) -- (11,0.7);
\draw[thick] (9.9,1.0) node[above,rotate=0] {$\sim {
    f_{1}}$};

\foreach \x/\xtext in {2.0/S_1,4.5/S_2,8/S_3,11/S_4}
\draw[shift={(\x,0)}] (0pt,2pt) -- (0pt,-2pt) node[below] {$\xtext$};
\end{tikzpicture}
\caption{\label{fig:X-process} Construction of a semi--Markov process based on Mittag--Leffler distributed interarrivals.}
\end{figure}
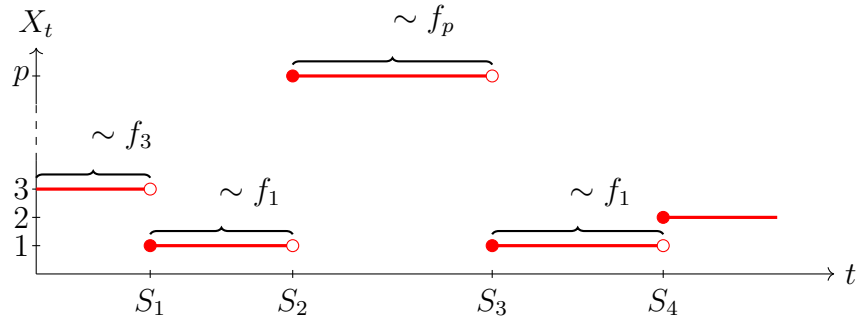

Define the intensity matrix $\mat{\Lambda}=\{ \lambda_{ij}\}_{i=1,...,p+1}$ by 
\[  \lambda_{ij}=\lambda_i q_{ij}, \ i\neq j, \ \ \mbox{and} \ \ \lambda_{ii}=-\lambda_i = \sum_{k\neq i}\lambda_{ik} ,\: i\le p ,  \]
and $\lambda_{p+1,i}=0$, and let 
\[  p_{ij}(t)  = \Prob (X_t=j | X_0=i), \ \ \mat{P}(t) = \{ p_{ij}(t) \}_{i,j=1,...,p} ,\]
be the probabilities that describe the dynamics of the process over the transient states. Then we may also write the matrix $\mat{\Lambda}$ in the following way 
\begin{equation}
 \mat{\Lambda} 
=
\begin{pmatrix}
\mat{T} & \vect{t} \\
\vect{0} & 0 
\end{pmatrix} . \label{eq:PH-structure-matrix_r}
\end{equation}

The matrix $\mat{\Lambda}$ can be associated with the intensity matrix for some Markov jump process. However, it is important to make the distinction that here we instead consider the semi-Markov process $X_t$, for which the dynamics on the transient states are not based on the exponential function but rather the Mittag-Leffler function, as is shown in the following result:

\begin{theorem}\label{transprobfrac} {\rm \cite{albrecher2019matrix}}
Let $\{ X_t\}_{t\geq 0}$ be the semi-Markov process constructed above. Then
\[   \mat{P}(t) = {E}_{\alpha,1}(\mat{T}t^\alpha) .  \] 
\end{theorem}

\noindent Define the Caputo derivative as the following fractional generalization of the ordinary differentiation operator,
$$
_{0}^{c} D_{t}^{\alpha} x(t)=\frac{1}{\Gamma(n-\alpha)} \int_{0}^{t}(t-\tau)^{n-\alpha-1} x^{(n)}(\tau) \dd \tau.
$$
Then Theorem \ref{transprobfrac} yields the following forward and backward type of Kolmogorov fractional differential equations:
\begin{corollary}
$_{0}^{c} D_{t}^{\alpha} \mat{P}(t)=\mat{T} \mat{P}(t)= \mat{P}(t)\mat{T}$.
\end{corollary}
\begin{proof}
It is well-known that the unique solution to the scalar fractional differential equation 
$$_{0}^{c} D_{t}^{\alpha} x(t)=a x(t), \quad t \geq 0,$$
is given in terms of the Mittag-Leffler function 
$$x(t)=E_{\alpha,1}\left(a t^{\alpha}\right).$$ The extension to the matrix case now follows from the representation \eqref{cauchyintegralML}:
\begin{align*}
_{0}^{c} D_{t}^{\alpha} \mat{P}(t)&=  \frac{1}{2\pi \ii}\oint_\gamma {_{0}^{c} D_{t}^{\alpha}}  E_{\alpha,\beta}(zt^\alpha )(z\mat{I}-\mat{T})^{-1}\dd z\\
&= \frac{1}{2\pi \ii}\oint_\gamma z  E_{\alpha,\beta}(zt^\alpha )(z\mat{I}-\mat{T})^{-1}\dd z ,
\end{align*}
but the latter equals both $\mat{T} \mat{P}(t)$ and $\mat{P}(t)\mat{T}$.
\end{proof}



\begin{theorem}{\rm \cite{albrecher2019matrix}}\label{Cor:MML-of-sample-path-model}
Let $\{X_t\}_{t\ge 0}$ be a semi-Markov process constructed as above, with $\mat{\Lambda}$ given by \eqref{eq:PH-structure-matrix_r}. 
Let $\tau=\inf\{t\ge 0:\, X_t= p+1\}$ denote the time until absorption. 
Then $\tau$ has a PH$_\alpha(\vect{\pi},\mat{T})$ distribution, with cumulative distribution function given by
$$
F_\tau(u)=1-\vect{\pi} {E}_{\alpha,1}(\mat{T} u^\alpha) \vect{e} .
$$
\end{theorem}

With this representation we are now ready to impose a reward structure on the different states of the process, thereby creating a dependent random vector in a way that extends the $\text{MPH}^\ast$ naturally.

For the absorption time $\tau$ as defined in Theorem \ref{Cor:MML-of-sample-path-model}, let now $r_{ik}$, $i=1,...,p$, $k=1,...,n$ be non--negative numbers and define
\[  Y_k = \int_0^\tau \sum_{i=1}^p r_{ik}1\{ X_t=i\} dt ,\ \ k=1,...,n.   \]
Form the column vectors $\vect{r}_k = (r_{1k},r_{2k},...,r_{pk})$, $k=1,...,n$, and matrix
\[ \mat{R} = ( \vect{r}_1,...,\vect{r}_n) . \]
The random variable $Y_k$ is interpreted as the total reward earned until absorption of $\{ X_t\}$, where $r_{ik}$ is the reward earned during sojourns in state $i$ of the variable $k$. Hence column $k$ of $\mat{R}$ defines a reward structure which defines variable $Y_k$. 
  See Figure \ref{MPH*} for a visual representation of the construction.

\begin{figure}[h]
\centering
\includegraphics[width=14cm,trim=8cm 8cm 8cm 8cm,clip]{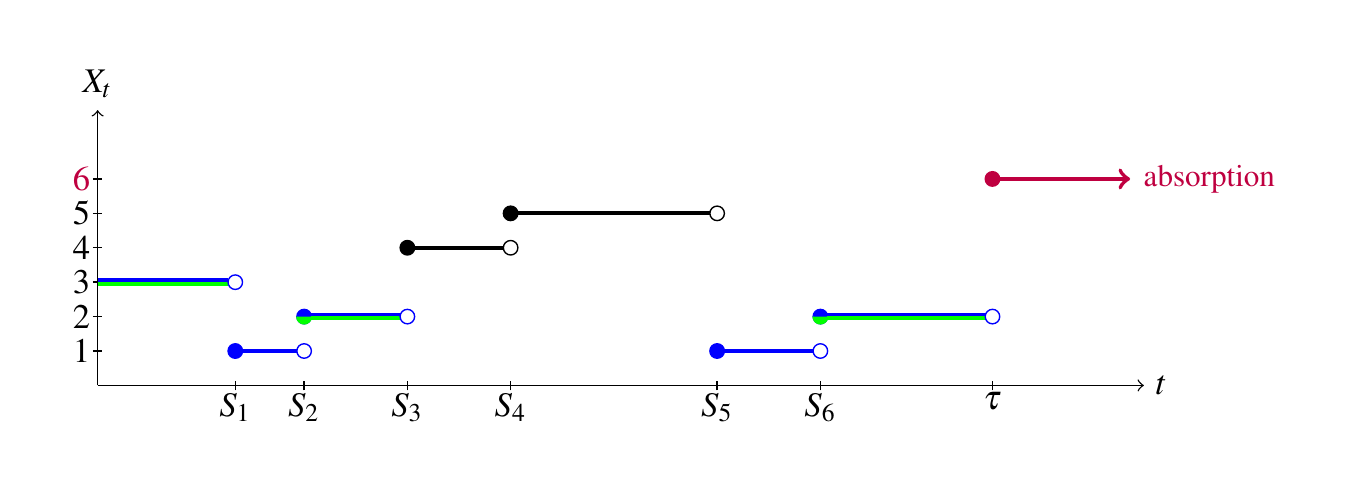}
\caption{Visual representation of the construction of the $\text{PH}_\alpha^\ast$ class. Here, three dimensions are considered: the first collects rewards during the blue holding times, corresponding to the first three states; the second during black holding times (independent of the first, in this case); and the third during green holding times (independent of the second, but not independent of the first).} 
\label{MPH*}
\end{figure}

We are interested in studying the joint distribution of 
 $\vect{Y}=(Y_1,....,Y_n)$. To this end, let $\vect{\theta}=(\theta_1,...,\theta_n)$ and 
 \[  H_i(\vect{\theta}) = \Exp \left. \left( e^{ -\inner{\vect{Y}}{\vect{\theta}}}\right| X_0=i \right) .  \]
 Condition on the first sojourn time $Z_{i1}$ in state $i$, which has a Mittag--Leffler distribution with parameters $(\lambda_i,\alpha)$. Let $\vect{Y}_{i1}$ denote the corresponding vector of rewards earned during 
 $[0,Z_{i1})$ and let $\vect{Y}_r$ denote the remaining rewards earned during $[Z_{i1},\tau)$. Then 
 $\vect{Y}=\vect{Y}_{i1}+\vect{Y}_r$ and $\vect{Y}_{i1}=Z_{i1} \vect{r}_i$. By the Markov renewal property,
 \begin{eqnarray*}
 \Exp\left. \left( e^{-\inner{\vect{Y}}{\vect{\theta}}} \right| X_0 =i  \right)&=& \Exp\left. \left( e^{-\inner{\vect{Y}_{i1}}{\vect{\theta}}} \right| X_0 =i  \right)\Exp\left. \left( e^{-\inner{\vect{Y}_r}{\vect{\theta}}} \right| X_0 =i, X_{Z_{i1}}  \right) \\
 &=& \Exp\left. \left( e^{-Z_{i1}\inner{\vect{r}_{i}}{\vect{\theta}}} \right| X_0 =i  \right)\Exp\left. \left( e^{-\inner{\vect{Y}_r}{\vect{\theta}}} \right| X_0 =i, X_{Z_{i1}}  \right)
 \end{eqnarray*}
Since $Z_{i1}$ is Mittag--Leffler distributed with parameters $(\lambda_i,\alpha)$, one gets 
\[  \Exp\left. \left( e^{-Z_{i1}\inner{\vect{r}_{i}}{\vect{\theta}}} \right| X_0 =i  \right) =  \frac{1}{1+ \inner{\vect{r}_i}{\vect{\theta}}^{\alpha} \lambda_i^{-1}}  . \]
 Recalling that $\mat{Q}=\{ q_{ij}\}$ contains the transition probabilities for the embedded Markov chain, we then have by a first step argument that
 \begin{eqnarray*}
 H_i(\vect{\theta})&=& \frac{1}{1+ \inner{\vect{r}_i}{\vect{\theta}}^{\alpha} \lambda_i^{-1}}\left(
q_{i,p+1} + \sum_{j\neq i} q_{ij} H_j(\vect{\theta}) 
    \right) .
 \end{eqnarray*}
 Using that $t_{ij}=\lambda_i q_{ij}$, $t_i=q_{i,p+1}\lambda_i$ we get that
 \[  \lambda_i H_i(\vect{\theta}) + \inner{\vect{r}_i}{\vect{\theta}}^{\alpha} H_i(\vect{\theta}) = t_i + \sum_{j\neq i} t_{ij} H_j(\vect{\theta}) \]
 which implies that
 \[  \inner{\vect{r}_i}{\vect{\theta}}^{\alpha} H_i(\vect{\theta}) = \sum_{j=1}^p t_{ij}H_j(\vect{\theta}) + t_i  .   \]
 In vector notation, with $\mat{\Delta}(\mat{R}\vect{\theta})^\alpha$ denoting the diagonal matrix which has $\inner{\vect{r}_i}{\vect{\theta}}^\alpha$, $i=1,...,p$, on its diagonal, we then write
 \begin{eqnarray*}
  \mat{\Delta}(\mat{R}\vect{\theta})^\alpha \mat{H}(\vect{\theta})&=& \mat{T}\mat{H}(\vect{\theta}) + \vect{t}
  \end{eqnarray*} 
  or
  \[ \mat{H}(\vect{\theta}) = \left(  \mat{\Delta}(\mat{R}\vect{\theta})^\alpha - \mat{T}  \right)^{-1}\vect{t}  . \]
  If $X_0 \sim \vect{\pi}$, we then get that the joint Laplace transform for $\vect{Y}$ is given by
  \begin{align}\label{LTmulti_rn}
   L_{\vect{Y}}(\vect{\theta}) = \vect{\pi} \left(  \mat{\Delta}(\mat{R}\vect{\theta})^\alpha - \mat{T}  \right)^{-1}\vect{t}  . 
   \end{align}
 \begin{definition}
 The joint distribution of rewards $\vect{Y}=(Y_1,...,Y_n)$, characterized by its Laplace transform \eqref{LTmulti_rn}, is said to have a multivariate fractional phase--type distribution, 
  and we shall denote it by 
  \[   \vect{Y} \sim \text{MPH}_\alpha^*(\vect{\pi},\mat{T},\mat{R}) . \]
  \end{definition}
\begin{remark}\rm \label{rem35}
Note that the only (yet subtle) difference between the Laplace transform of the GMML distribution introduced in  \cite[Eq.15]{multiml} and the corresponding expression  \eqref{LTmulti_rn} above is that the power $\alpha$ is applied after and not before the left-multiplication with the reward matrix. One consequence is that the scalar parameter $\alpha$ represents the regular variation index for all marginals alike, in contrast to the GMML construction in \cite{multiml}, where different values were possible for each component. However, the extension to powers as described in Section \ref{sec:dense} allows to alleviate that issue when desirable.\end{remark}

\subsection{Denseness properties of the $\text{MPH}_\alpha^\ast$ class and an extension}\label{sec:dense}
As members of $\text{PH}_\alpha$, the marginals of the $\text{MPH}_\alpha^\ast$ class all have regularly varying tails with (the same) index $\alpha< 1$ (which in particular entails an infinite mean). In order to allow for more flexibility, a simple extension is to consider (possibly different) powers of each random component, which leads to arbitrary positive index of regular variation for each component.\\
 Let $\vect{X}\sim \mbox{MPH}_\alpha^*(\vect{\pi},\mat{T},\mat{R})$ with density $f_{\vect{X}}(x_1,...,x_n)$. Let $\vect{\nu}=(\nu_1,...,\nu_n)$ for $\nu_i>0$, $i=1,...,n$ and consider the transformed random vector
\[  \vect{Y} = \vect{X}^{1/\vect{\nu}}=(X_1^{1/\nu_1},...,X_n^{1/\nu_n}) ,   \]
for which the joint density is given by 
\[  f_{\vect{Y}}(y_1,...,y_n)=\left( \prod_{i=1}^n \nu_i y_i^{\nu_i-1}\right) f_{\vect{X}}(y_1^{\nu_1},...,y_n^{\nu_n}).  \]
We refer to this enlarged class as the $\text{MPH}_\alpha^{\ast1/\nu}$ class. Then we have the following result:
\begin{theorem}
	{\ }
	\begin{enumerate}[(i)]
		\item The class $\mbox{MPH}_\alpha^\ast(\vect{\pi},\mat{T},\mat{R})$ is dense in the class of distributions on $\mathbb{R}_+^n$.\\
		\item For any fixed  ${\alpha}$, the class $\text{MPH}_\alpha^{\ast1/\vect{\nu}}(\vect{\pi},\mat{T},\mat{R})$ is dense in the class of distributions on $\mathbb{R}_+^n$.\\
		\item For any fixed vector of positive tail indices $(\alpha/\nu_1,...,\alpha/\nu_n)$, the class \\ $\text{MPH}_\alpha^{\ast1/\vect{\nu}}(\vect{\pi},\mat{T},\mat{R})$ is dense in the class of distributions on $\mathbb{R}_+^n$.\end{enumerate}
\end{theorem}
\begin{proof}
In Section \ref{sec:feedforward}, it will be shown that for the particular subclass of feed-forward type with transition matrix \eqref{tmm} and reward matrix \eqref{rmm} the identity \eqref{iden} holds and therefore the $\text{GMML}$ and $\text{PH}_\alpha^\ast$ classes agree in that particular case. For this particular structure, the phase--type case $\alpha=1$ is still dense on $\mathbb{R}_+^n$, but then the proof of all three items above follows along the same lines as Theorem 4.10 in \cite{multiml}.
\end{proof}

\noindent Notice that (iii) in particular shows that we can approximate any distribution on $\mathbb{R}_+^n$ arbitrarily closely through distributions in  $\text{MPH}_\alpha^{\ast1/\vect{\nu}}(\vect{\pi},\mat{T},\mat{R})$ with a pre-specified regularly varying index for each marginal.

\subsection{A product representation}
We proceed to show a representation theorem which sheds some light on the dependence structure of the $\text{MPH}_\alpha^\ast$ class. 

\begin{theorem}\label{repPHstable2}
	Let  $\vect{Y}$ have Laplace transform \eqref{LTmulti_rn}. Then
	\begin{align}\label{prodrepresentation_r2}
	\vect{Y}\stackrel{d}{=}\mat{R}^T\vect{W}^{1/\alpha}\bullet \vect{S}_\alpha,
	\end{align}
	where $\vect{W}^{1/\alpha}=(W_1^{1/\alpha},\dots,W_n^{1/\alpha})$ with $\vect{W}=(W_1,...,W_n)\sim \mbox{MPH}^*(\vect{\pi
	},\mat{T},\mat{I})$ (see \eqref{eq:Kulkarni_r}), and where  $\vect{S}_\alpha=(S^1_{\alpha},\dots,S^n_{\alpha})$ is a vector of independent stable random variables, each with Laplace transform $\exp(-u^{\alpha})$. Here, $\bullet$ refers to Schur (or entry-wise) multiplication of vectors.
\end{theorem}
\begin{proof}
	We first recall that for generic vectors $\vect{u},\vect{v}$ we have
	\begin{align*}
	\langle \mat{R}\vect{u},\vect{v} \rangle=\langle\vect{u},\mat{R}^{T}\vect{v} \rangle, 
	\end{align*}
	from which
	\begin{align*}
	\E(\exp(-\langle\vect{u},\mat{R}^T\vect{W}^{1/\alpha}\bullet \vect{S}_\alpha \rangle))&=
	\int_{\mathbb{R}_+^n} \E(\exp(-\langle\mat{R}\vect{u},\vect{w}^{1/\alpha}\bullet \vect{S}_\alpha \rangle)) \dd F_{\vect{W}}(\vect{w}) \\
	&=
	\int_{\mathbb{R}_+^n} \exp(-[(\mat{R}\vect{u})_1^{\alpha}w_1+\dots+(\mat{R}\vect{u})_n^{\alpha}w_n]) \dd F_{\vect{W}}(\vect{w}) \\
	&=
	\int_{\mathbb{R}_+^n} \exp(-\langle (\mat{R}\vect{u})^{{\alpha}},\vect{w}\rangle) \dd F_{\vect{W}}(\vect{w}) \\
	&=\vect{\pi}\left(  \mat{\Delta}(\mat{R}\vect{u})^{{\alpha}} - \mat{T} \right)^{-1}\vect{t}.
	\end{align*}
\end{proof}
The above result gives insight into how tail dependence is created (in contrast to the analogous Theorem 6 in \cite{multiml}): the reward matrix $\mat{R}$ determines how the a priori independent stable components $S^i_{\alpha}$ are combined towards tail-dependent components $Y_i,\; (i=1,\ldots,n)$, with tail dependence asymptotically being concentrated on lines whith slopes governed by $\mat{R}$.

\subsection{Distribution of projections}\label{ren_proj_sec}
Consider 
$ \vect{Y} \sim \text{MPH}_\alpha^*(\vect{\pi},\mat{T},\mat{R}) $ 
with Laplace transform \eqref{LTmulti_rn}. 
We are interested in the distribution of the linear combination $\inner{\vect{Y}}{\vect{w}}$ of the components for some non--zero, non--negative vector $\vect{w}$. Split the state space of $E$ in $E_+$ and $E_0$ according to whether $(\mat{R}\vect{w})_i$ is positive or zero, respectively, and decompose $\vect{\pi}=(\vect{\pi}_+,\vect{\pi}_0)$ and 
\[  \mat{T} =
\begin{pmatrix}
\mat{T}_{++} & \mat{T}_{+0} \\
\mat{T}_{0+} & \mat{T}_{00}
\end{pmatrix}  \]  
accordingly. Then consider the Laplace transform for $\inner{\vect{Y}}{\vect{w}}$, 
which is
\begin{eqnarray*}
	\Exp \left( e^{-u \inner{\vect{Y}}{\vect{w}} }   \right) &=& 
	\Exp \left( e^{-\inner{\vect{Y}}{u\vect{w}}}\right) \\
	&=& \vect{\pi} \left(  \mat{\Delta}((\mat{R}u\vect{w})^\alpha) - \mat{T}  \right)^{-1}\vect{t} \\
	&=&\vect{\pi} \left( u^\alpha \mat{\Delta}((\mat{R}\vect{w})^\alpha) - \mat{T}  \right)^{-1}\vect{t} \\
	&=& \vect{\pi}
	\begin{pmatrix}
		u^\alpha \mat{\Delta}((\mat{R}\vect{w})^\alpha)_+-\mat{T}_{++} & - \mat{T}_{+0} \\
		-\mat{T}_{0+} & -\mat{T}_{00}
	\end{pmatrix}^{-1}
	\vect{t} \\
	&=&(\vect{\pi}_+,\vect{\pi}_0)
	\begin{pmatrix}
		\mat{A}_{11} & \mat{A}_{12} \\
		\mat{A}_{21} & \mat{A}_{22}
	\end{pmatrix}
	\begin{pmatrix}
		\vect{t}_+ \\
		\vect{t}_0
	\end{pmatrix},
\end{eqnarray*}
where
\begin{eqnarray*}
	\mat{A}_{11}&=&\left( u^\alpha \mat{\Delta}((\mat{R}\vect{w})^\alpha)_+ -\mat{T}_{++} - 
	\mat{T}_{+0}(-\mat{T}_{00})^{-1}\mat{T}_{0+}
	\right)^{-1} \\
	&=& \bigg( u^\alpha \mat{I} - \mat{\Delta}((\mat{R}\vect{w})^\alpha)_+^{-1} \left[ \mat{T}_{++} + 
	\mat{T}_{+0}(-\mat{T}_{00})^{-1}\mat{T}_{0+} \right]
	\bigg)^{-1} \mat{\Delta}((\mat{R}\vect{w})^\alpha)_+^{-1} \\
	&=& \left( u^\alpha \mat{I} - \mat{T}_{\vect{w}^{\vect{\alpha}}}
	\right)^{-1} \mat{\Delta}((\mat{R}\vect{w})^\alpha)_+^{-1}, \\
	& & \\
	\mat{A}_{12}&=&\left( u^\alpha \mat{I} - \mat{T}_{\vect{w}^{\vect{\alpha}}}
	\right)^{-1} \mat{\Delta}(\mat{R}\vect{w}^{\vect{\alpha}})_+^{-1}\mat{\Delta}((\mat{R}\vect{w})^\alpha)_+^{-1}, \\
	& & \\
	\mat{A}_{21}&=& (-\mat{T}_{00})^{-1}\mat{T}_{0+} \left( \mat{\Delta}(u^\alpha \mat{I} - \mat{T}_{\vect{w}^{\vect{\alpha}}}
	\right)^{-1} \mat{\Delta}((\mat{R}\vect{w})^\alpha)_+^{-1}, \\
	& & \\
	\mat{A}_{22}&=&(-\mat{T}_{00})^{-1}\left( \mat{I} + \mat{T}_{0+}\left( u^\alpha \mat{I} - \mat{T}_{\vect{w}^{\vect{\alpha}}}
	\right)^{-1} \mat{\Delta}((\mat{R}\vect{w})^\alpha)_+^{-1}\mat{T}_{+0}(-\mat{T}_{00})^{-1} \right) 
\end{eqnarray*}
and
\[ 
\boldsymbol{T}_{w}=\boldsymbol{\Delta}\left((\boldsymbol{R} \boldsymbol{w})_{+}^\alpha \right)^{-1}\left(\boldsymbol{T}_{++}+\boldsymbol{T}_{+0}\left(-\boldsymbol{T}_{00}\right)^{-1} \boldsymbol{T}_{0+}\right)  . \]
Let 
\[ \boldsymbol{\pi}_{w}=\boldsymbol{\pi}_{+}+\boldsymbol{\pi}_{0}\left(-\boldsymbol{T}_{00}\right)^{-1} \boldsymbol{T}_{0+} .  \]
Then
\begin{eqnarray*}
	\vect{\pi}_+\mat{A}_{11}+\vect{\pi}_0\mat{A}_{21}&=&
	\vect{\pi}_{\vect{w}}\left( u^\alpha \mat{I} - \mat{T}_{\vect{w}}
	\right)^{-1} \mat{\Delta}((\mat{R}\vect{w})^\alpha)_+^{-1}, \\
	\vect{\pi}_+\mat{A}_{12}+\vect{\pi}_0\mat{A}_{22}&=&\vect{\pi}_0(-\mat{T}_{00})^{-1} + 
	\vect{\pi}_{\vect{w}}\left(u^\alpha \mat{I} - \mat{T}_{\vect{w}}
	\right)^{-1}\mat{\Delta}((\mat{R}\vect{w})^\alpha)_+^{-1}\mat{T}_{+0}(-\mat{T}_{00})^{-1} .
\end{eqnarray*}
Now inserting
\[ \begin{pmatrix}
\vect{t}_+ \\
\vect{t}_0
\end{pmatrix} = 
-\mat{T}\vect{e} = \begin{pmatrix}
-\mat{T}_{++}\vect{e}-\mat{T}_{+0}\vect{e} \\
-\mat{T}_{0+}\vect{e}-\mat{T}_{00}\vect{e} 
\end{pmatrix},
\]
we get 
\begin{eqnarray*}
	\lefteqn{\left(\vect{\pi}_+\mat{A}_{11}+\vect{\pi}_0\mat{A}_{21}\right) \vect{t}_+ + 
		\left(\vect{\pi}_+\mat{A}_{12}+\vect{\pi}_0\mat{A}_{22} \right)\vect{t}_0 }~~~~\\
	&=&\vect{\pi}_0 (\mat{I} - (-\mat{T}_{00})^{-1}\mat{T}_{0+})\vect{e} + 
	\vect{\pi}_{\vect{w}} \left( u^\alpha \mat{I} - \mat{T}_{\vect{w}}
	\right)^{-1}\vect{t}_{\vect{w}} \\
	&=& 1-\boldsymbol{\pi}_{w}\vect{e} + \vect{\pi}_{\vect{w}} \left(u^\alpha \mat{I} - \mat{T}_{\vect{w}}
	\right)^{-1}\vect{t}_{\vect{w}}
\end{eqnarray*}
with
\[ \vect{t}_{\vect{w}} = -\mat{T}_{\vect{w}} \vect{e}.  \]
Thus we have proved the following result. 

\begin{theorem}
	Let $ \vect{Y} \sim \text{MPH}_\alpha^*(\vect{\pi},\mat{T},\mat{R}) $ and $\vect{w}\geq \vect{0}$ be a non--zero vector. Then $\inner{\vect{Y}}{\vect{w}}$ has a distribution with an absolutely continuous part being  $\mbox{PH}_\alpha(\vect{\pi}_{\vect{w}},\mat{T}_{\vect{w}})$ distributed, where 
	\begin{eqnarray*}
		\boldsymbol{\pi}_{w}&=&\boldsymbol{\pi}_{+}+\boldsymbol{\pi}_{0}\left(-\boldsymbol{T}_{00}\right)^{-1} \boldsymbol{T}_{0+} \\
		\boldsymbol{T}_{w}&=&\boldsymbol{\Delta}\left((\boldsymbol{R} \boldsymbol{w})_{+}^\alpha \right)^{-1}\left(\boldsymbol{T}_{++}+\boldsymbol{T}_{+0}\left(-\boldsymbol{T}_{00}\right)^{-1} \boldsymbol{T}_{0+}\right)  
	\end{eqnarray*}
	and an atom at zero of size $1-\boldsymbol{\pi}_{w}\vect{e}$.
\end{theorem}
As a simple consequence of the above result, one can retrieve the form of the marginal distributions for any choice of $\boldsymbol{T}$ and $\boldsymbol{R}$:
\begin{corollary}
	Let $\tau \sim \text{PH}_\alpha(\vect{\pi},\mat{T})$. Let $\vect{r}=(r_1,...,r_p)$ be a non--zero non--negative vector of rewards. Let $\{ X_t \}_{t\geq 0}$ denote the semi--Markov process which generates $\tau$ and define
	\[  Y = \int_0^\tau \sum_{i=1}^p r_i 1\{ X_t =i \}dt   \]
	which is the total reward earned up to time $\tau$. Then $Y$ has a 
	distribution with an absolutely continuous part having a $\mbox{PH}_\alpha(\tilde{\vect{\pi}},\tilde{\mat{T}})$ form, 
	where
	\begin{eqnarray*}
		\tilde{\vect{\pi}}&=& \boldsymbol{\pi}_{+}+\boldsymbol{\pi}_{0}\left(-\boldsymbol{T}_{00}\right)^{-1} \boldsymbol{T}_{0+} \\
		\tilde{\mat{T}}&=& \boldsymbol{\Delta}\left(\vect{r}_{+}^\alpha \right)^{-1}\left(\boldsymbol{T}_{++}+\boldsymbol{T}_{+0}\left(-\boldsymbol{T}_{00}\right)^{-1} \boldsymbol{T}_{0+}\right) 
	\end{eqnarray*}  
	and an atom at zero of size $1-\tilde{\vect{\pi}}\vect{e}$.
\end{corollary} 
\begin{remark}\rm In case all rewards are strictly positive, the translation between $\mbox{PH}_\alpha$ distributions with and without rewards is even simpler: Consider the process $\{ X_t\}_{t\geq 0}$ defined in \eqref{def:X-process_r} underlying a $\mbox{PH}_\alpha(\vect{\pi},\mat{T})$ distribution,
	and assume that a reward $r_i>0$ is earned when the process is in state $i$, $i=1,2,...,p$. Then the total reward earned up to the time of absorption is  $\mbox{PH}_\alpha(\vect{\pi},\mat{S})$ distributed  with 
	\[   \mat{S} = \mat{\Delta}(\vect{r}^{-\alpha})\mat{T} ,  \]
	where $\vect{r}^{-\alpha}=(r_1^{-\alpha},...,r_p^{-\alpha})$. Hence a reward of rate $r_i$ in state $i$ may be achieved by dividing row $i$ of $\mat{T}$ by $r_i^{\alpha}$. This can also be seen directly from the construction of the semi-Markov process, since the $\lambda_i$ are scale parameters.
\end{remark}

\section{Two specific examples}\label{sec:spec}
\subsection{The feed-forward case}\label{sec:feedforward}
The $\text{MPH}_\alpha^*$ class shares an important sub-class of distributions with the GMML class introduced in \cite{multiml}. The so-called \textit{feed--forward} sub-class is based on a special structure of the matrix components given as follows.

Let $\mat{C}_1,...,\mat{C}_n$ be sub--intensity matrices and let $\mat{D}_1,...,\mat{D}_n$ be non--negative matrices such that $-\mat{C}_i\vect{e}=\mat{D}_i\vect{e}$. Define the initial vector as
$\vect{\beta}=(\vect{\pi},\vect{0},...,\vect{0})$ 
and the matrix
\begin{align}\label{tmm}
\mat{T}=
\begin{pmatrix}
\mat{C}_1 & \mat{D}_1 & \mat{0} & \cdots & \mat{0} \\
\mat{0} &  \mat{C}_2 & \mat{D}_2 & \cdots & \mat{0} \\
\mat{0} & \mat{0} & \mat{C}_3 & \cdots & \mat{0} \\
\vdots & \vdots & \vdots & \vdots\vdots\vdots & \vdots \\
\mat{0} & \mat{0} &\mat{0} & \cdots & \mat{C}_n.
\end{pmatrix} .
 \end{align}
The reward matrix consists of
 \begin{align}\label{rmm}
 \mat{R}= \begin{pmatrix}
 \vect{e} & \vect{0} & \vect{0} & \cdots & \vect{0} \\
 \vect{0} & \vect{e} & \vect{0} & \cdots & \vect{0} \\
 \vect{0} & \vect{0} & \vect{e} & \cdots & \vect{0} \\
 \vdots & \vdots & \vdots & \vdots\vdots\vdots & \vdots \\
 \vect{0} & \vect{0} & \vect{0} & \cdots & \vect{e}
 \end{pmatrix} . 
 \end{align}

\noindent In this case it is immediate that \begin{equation}
\mat{\Delta}\left(\mat{R}\vect{\theta}\right)^\alpha=\mat{\Delta}\left(\mat{R}\vect{\theta}^\alpha\right)\label{iden}
\end{equation} in which case the respective distributions in the GMML and MPH$_\alpha^*$ classes coincide (cf.\ Remark \ref{rem35}). Correspondingly, the explicit forms of the Laplace transform and density can be found in Theorem 8 of \cite{multiml} (choosing  $\alpha_1=\cdots=\alpha_n=\alpha$ for the present context). 
%
%

\noindent Note, however, that in general we do not have $\text{GMML}\subset \text{MPH}_\alpha^\ast$ nor that $\text{GMML} \supset \text{MPH}_\alpha^\ast$ (keeping in mind that the $\text{GMML}$ class contains distributions with possibly different tail index in each marginal and no possible tail dependence, whereas the $\text{MPH}_\alpha^\ast$ class contains distributions with the same tail index for the marginals, but possible tail dependence), see also Figure \ref{mapofdists}. 

\subsection{A two-dimensional explicit example with tail dependence}
Suppose that $\vect{X}=(X_1,X_2)\sim  \text{MPH}_\alpha^*(\vect{\pi},\mat{T},\mat{R})$, where 
\[  \vect{\pi}=(\vect{\pi}_1,\vect{\pi}_2,\vect{\pi}_3), \ \ \mat{T} = 
\begin{pmatrix}
\mat{T}_{11} & \mat{T}_{12} & \mat{T}_{13} \\
\mat{0} & \mat{T}_{22} & \mat{0} \\
\mat{0} & \mat{0} & \mat{T}_{33}
\end{pmatrix}, \ \ \mbox{and} \ \ \mat{R} =
\begin{pmatrix}
\vect{e} & \vect{e} \\
\vect{e} & \vect{0} \\
\vect{0} & \vect{e} 
\end{pmatrix},
 \]
 $\vect{\pi}_i$ are $p_i$--dimensional vectors and $\mat{T}_{ij}$ are $p_i\times p_j$ --dimensional matrices for $i=1,2,3$. As usual we let
 \[  \vect{t} = (\vect{t}_1,\vect{t}_2,\vect{t}_3)^\prime = -\mat{T}\vect{e} . \]
 Hence $\vect{t}_i$ is the vector of rates for jumping to the absorbing state from block $i=1,2,3$. Denote the set of transient states by $E=\{1,2,....,p_1+p_2+p_3\}$ and let 
 $E_1=\{1,2,...,p_1\}$ denote the states corresponding to the first block, $E_2 = \{ p_1+1,...,
 p_1+p_2\}$ the states of the second block and $E_3=\{p_1+p_2+1,...,p_1+p_2+p_3\}$ the states of the third block.

 The joint density function of the underlying multivariate phase--type distribution $\vect{X}=(X_1,X_2)\sim \mbox{MPH}^*(\vect{\pi},\mat{T},\mat{R})$ is given by (see \cite[p.448]{bladt2017matrix})
 \[ f\left(x_{1}, x_{2}\right)=\left\{\begin{aligned} \boldsymbol{\pi}_{1} e^{\boldsymbol{T}_{11} x_{2}} \boldsymbol{T}_{12} e^{\boldsymbol{T}_{22}\left(x_{1}-x_{2}\right)} 
 \boldsymbol{t}_{2}, & &  0<x_{2}<x_{1} \\
  \boldsymbol{\pi}_{1} e^{\boldsymbol{T}_{11} x_{1}} \boldsymbol{T}_{13} e^{\boldsymbol{T}_{33}\left(x_{2}-x_{1}\right)} \boldsymbol{t}_{3}, & & 0<x_{1}<x_{2} \\
   \boldsymbol{\pi}_{1} e^{\boldsymbol{T}_{11} x_{1}} \boldsymbol{t}_{1}, & &x_{1}=x_{2} \\
   \boldsymbol{\pi}_{2} e^{\boldsymbol{T}_{22} x_{1}} \boldsymbol{t}_{2}, & &x_{1}>0, x_{2}=0 \\ 
   \boldsymbol{\pi}_{3} e^{\boldsymbol{T}_{33} x_{2}} \boldsymbol{t}_{3}, & &x_{1}=0, x_{2}>0.
   \end{aligned}\right.  \]
There is a component of sharing rewards in this structure. If the Markov jump process is started in a state in $E_1$, then reward is earned for both variables $X_1$ and $X_2$, and
if $\vect{t}_1\neq \vect{0}$, then there is a positive probability that the underlying process will exit to the absorbing state directly from the Block 1, in which case $X_1=X_2$. 

We shall now consider the distribution of $\vect{Y}$ with Laplace transform \eqref{LTmulti_rn}. First we notice that 
\[ \mat{\Delta}(\mat{R}\vect{\theta})^\alpha =
\begin{pmatrix}
(\theta_1+\theta_2)^\alpha \mat{I} & \mat{0} & \mat{0} \\
\mat{0} & \theta_1^\alpha \mat{I} & \mat{0} \\
\mat{0} & \mat{0} & \theta_2^\alpha \mat{I}
\end{pmatrix},
 \]
 where the dimensions of the identity matrices $\mat{I}$ are $p_1$, $p_2$ and $p_3$, respectively. Let 
 \begin{eqnarray*}
  \mat{A}_{11}&=&\left((\theta_1+\theta_2)^\alpha \mat{I}-\mat{T}_{11}\right)^{-1}=\int_0^\infty  e^{-(\theta_1+\theta_2)x}x^{\alpha-1}
 {E}_{\alpha,\alpha}(\mat{T}_{11}x^\alpha)\ dx\\
\mat{A}_{22}&=&\left(\theta_1^\alpha \mat{I}-\mat{T}_{22}\right)^{-1}=\int_0^\infty e^{-\theta_1 y}y^{\alpha-1}{E}_{\alpha,\alpha}(\mat{T}_{22}y^{\alpha-1})\ dy \\
\mat{A}_{33}&=&\left(\theta_2^\alpha \mat{I}-\mat{T}_{33}\right)^{-1}=\int_0^\infty e^{-\theta_2 y}y^{\alpha-1}x{E}_{\alpha,\alpha}(\mat{T}_{33}y^{\alpha-1})\ dy .
  \end{eqnarray*} 
Then
\[  \left( \mat{\Delta}(\mat{R}\vect{\theta})^\alpha - \mat{T} \right)^{-1}
=
\begin{pmatrix}
\mat{A}_{11} & \mat{A}_{11}\mat{T}_{12}\mat{A}_{22} & \mat{A}_{11}\mat{T}_{13}\mat{A}_{33} \\
\mat{0} & \mat{A}_{22} & \mat{0} \\
\mat{0} & \mat{0} & \mat{A}_{33} 
\end{pmatrix} .
 \]
 Hence
 \begin{eqnarray*}
 L_{\vect{X}}(\vect{\theta})&=&\vect{\pi} \left( \mat{\Delta}(\mat{R}\vect{\theta})^\alpha - \mat{T} \right)^{-1} \vect{t} \\
 &=&\vect{\pi}_1 \mat{A}_{11}\vect{t}_1 + \vect{\pi}_2 \mat{A}_{22}\vect{t}_2 + \vect{\pi}_3 \mat{A}_{33}\vect{t}_3+ \vect{\pi}_1\mat{A}_{11}\mat{T}_{12}\mat{A}_{22}\vect{t}_2 + \vect{\pi}_1\mat{A}_{11}\mat{T}_{13}\mat{A}_{33}\vect{t}_3 .
 \end{eqnarray*}
 The term 
 \begin{eqnarray*}
 \lefteqn{\vect{\pi}_1\mat{A}_{11}\mat{T}_{12}\mat{A}_{22}\vect{t}_2}~~\\
 &=&\vect{\pi}_1\int_0^\infty  e^{-(\theta_1+\theta_2)x}x^{\alpha-1}
 {E}_{\alpha,\alpha}(\mat{T}_{11}x^\alpha)\ dx \ \mat{T}_{12}\int_0^\infty e^{-\theta_1 y}y^{\alpha-1}{E}_{\alpha,\alpha}(\mat{T}_{22}y^{\alpha-1})\ dy\ \vect{t}_2 \\
 &=&\vect{\pi}_1\int_0^\infty \int_0^\infty e^{-(\theta_1+\theta_2)x-\theta_1y}x^{\alpha-1}y^{\alpha-1} {E}_{\alpha,\alpha}(\mat{T}_{11}x^\alpha)\mat{T}_{12}{E}_{\alpha,\alpha}(\mat{T}_{22}y^\alpha)\ dx \ dy \ \vect{t}_2 \\
 &=&\vect{\pi}_1\int_0^\infty \int_0^\infty e^{-\theta_1 (x+y) - \theta_2 x}x^{\alpha-1}y^{\alpha-1} {E}_{\alpha,\alpha}(\mat{T}_{11}x^\alpha)\mat{T}_{12}{E}_{\alpha,\alpha}(\mat{T}_{22}y^\alpha)\ dy \ dx \ \vect{t}_2 \\
 &=& \vect{\pi}_1\int_0^\infty \int_x^\infty e^{-\theta_1 z - \theta_2 x} x^{\alpha-1}(z-x)^{\alpha-1}  {E}_{\alpha,\alpha}(\mat{T}_{11}x^\alpha)\mat{T}_{12}{E}_{\alpha,\alpha}(\mat{T}_{22}(z-x)^\alpha )\ dz \ dx \vect{t}_2 \\
 &=& \vect{\pi}_1\int_0^\infty e^{-\theta_2 x} x^{\alpha-1} {E}_{\alpha,\alpha}(\mat{T}_{11}x^\alpha)\mat{T}_{12} \int_x^\infty e^{-\theta_1 z} (z-x)^{\alpha-1}{E}_{\alpha,\alpha}(\mat{T}_{22}(z-x)^\alpha ) dz \ dx \ \vect{t}_2 ,
 \end{eqnarray*}
 which is hence the Laplace transform for the joint density of the form
 \[  \vect{\pi}_1 x^{\alpha-1}{E}_{\alpha,\alpha}(\mat{T}_{11}x^\alpha)\mat{T}_{12} (y-x)^{\alpha-1}
{E}_{\alpha,\alpha}(\mat{T}_{22}(y-x)^\alpha) \vect{t}_2
  \]
  when $Y_1>Y_2$. A similar argument applies to 
  $\vect{\pi}_1\mat{A}_{11}\mat{T}_{13}\mat{A}_{33}\vect{t}_3$. The terms 
$\vect{\pi}_2 \mat{A}_{22}\vect{t}_2$ and $\vect{\pi}_3 \mat{A}_{33}\vect{t}_3$ correspond to the Laplace transform where one of the variables is equal to zero, while the term $\vect{\pi}_1 \mat{A}_{11}\vect{t}_1 $ corresponds to the joint Laplace transform when $Y_1=Y_2$. In conclusion,
\[ f_{\vect{Y}}\left(x,y\right)=\left\{
\begin{aligned} 
\vect{\pi}_1 y^{\alpha-1}{E}_{\alpha,\alpha}(\mat{T}_{11}y^\alpha)\mat{T}_{12} (x-y)^{\alpha-1}
{E}_{\alpha,\alpha}(\mat{T}_{22}(x-y)^\alpha ) \vect{t}_2, & &  0<y<x  \\
\vect{\pi}_1 x^{\alpha-1}{E}_{\alpha,\alpha}(\mat{T}_{11}x^\alpha)\mat{T}_{13} (y-x)^{\alpha-1}
{E}_{\alpha,\alpha}(\mat{T}_{33}(y-x)^\alpha) \vect{t}_2, & &  0<x<y \\
  x^{\alpha-1} \boldsymbol{\pi}_{1}  {E}_{\alpha,\alpha} (\boldsymbol{T}_{11}x^\alpha)  \boldsymbol{t}_{1}, & &x=y \\
   x^{\alpha-1} \boldsymbol{\pi}_{2}  {E}_{\alpha,\alpha} (\boldsymbol{T}_{22}x^\alpha)  \boldsymbol{t}_{2},  & &x>0, y=0 \\ 
   x^{\alpha-1} \boldsymbol{\pi}_{3}  {E}_{\alpha,\alpha} (\boldsymbol{T}_{33}x^\alpha)  \boldsymbol{t}_{3},  & &x=0, y>0.
   \end{aligned}\right.  \]
 An atom at zero (with point mass $1-\vect{\pi}\vect{e}$) could also have been achieved for both cases by letting $\vect{\pi}\vect{e}<1$. Figure \ref{bivariateMML*} depicts a corresponding density, along with simulated data from the same distribution. The parameters are chosen to be $\alpha=0.9$, $\vect{\pi}_1=(1/2,1/2)$, $\vect{\pi}_2=\vect{\pi}_3=\vect{0}$, and 
\begin{align*}
\mat{T}_{11}=\begin{pmatrix}
 -3 & 2   \\
0 &-4
 \end{pmatrix}, \quad
 \mat{T}_{12}=\mat{T}_{13}=\begin{pmatrix}
 0 & 1/2   \\
1 & 1 
 \end{pmatrix},\quad
 \mat{T}_{22}=\mat{T}_{33}\begin{pmatrix}
 -1 & 1   \\
0 &-2
 \end{pmatrix},
 \end{align*} 
which implies that there is no mass at $x=0$, $y=0$, or $x=y$. One clearly observes the resulting tail dependence across the respective slopes. 

\begin{figure}[hh]
\centering
\includegraphics[width=14cm,trim=11cm 2cm 2cm .5cm,clip]{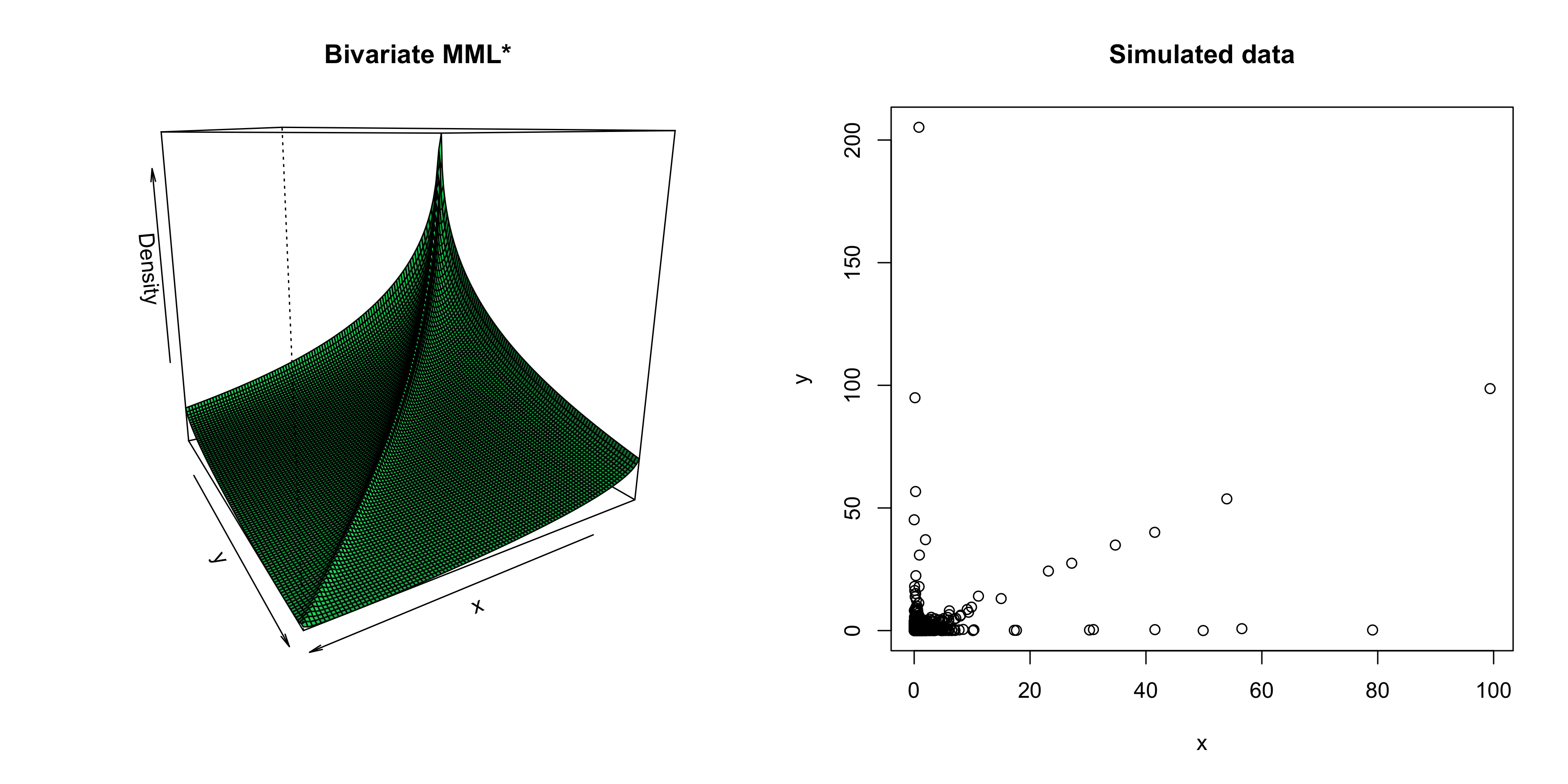}
\caption{Density and $1500$ simulated data from a bivariate MPH$_\alpha^*$ distribution.} 
\label{bivariateMML*}
\end{figure}

\section{Conclusion}\label{concl}

In this paper we propose an extension of Kulkarni's construction method to define a new class of multivariate distributions with matrix Mittag-Leffler distributed marginals. Based on a time-fractional sample path approach of an underlying semi-Markov jump process, this new class allows for dependence in the tails, yet still a rather explicit representation. We work out in detail how this class complements an earlier construction of a multivariate Mittag-Leffler distribution in \cite{multiml}. The main contribution of this paper is on the conceptual and mathematical side. It will be interesting in future research to complement the present contribution by developing fitting procedures for real multivariate data sets in applications, which exploit the explicit expressions obtained for this new class and study its versatility in more detail. It will also be challenging to study procedures that decide about the appropriate dimensions of the underlying matrices in concrete applications. 
%

\bibliographystyle{plain}
\bibliography{renewal-2}

\end{document}